\documentclass{article}

\usepackage[usenames,dvipsnames]{color} 
\usepackage{etex}
\usepackage{graphicx,amssymb,amsmath,amsthm}
\usepackage{epsfig}
\usepackage{pstricks}

\usepackage{tikz}
\usepackage{soul}   
\usepackage{xy} \input xy \xyoption{all}
\usepackage{algorithmic}
\usepackage{cancel}
\usepackage{soul}   
\usepackage{caption}
\usepackage{subcaption}
\usepackage{enumitem}
\usepackage{hyperref}
\hypersetup{colorlinks, citecolor=blue, filecolor=blue, linkcolor=blue, urlcolor=blue} 

\newcounter{mnotecounter}

\newtheorem{theorem}{Theorem}[section]
\newtheorem{lemma}[theorem]{Lemma}
\newtheorem{definition}[theorem]{Definition}
\newtheorem{proposition}[theorem]{Proposition}
\newtheorem{example}[theorem]{Example}
\newtheorem{remark}[theorem]{Remark}
\newtheorem{corollary}[theorem]{Corollary}

\def\qed{\hfill\vbox{\hrule\hbox{\vrule\kern3pt\vbox{\kern6pt}\kern3pt\vrule}\hrule}\bigskip}

\def\P{{\mathbb P}}

\newcommand{\F}{\mathbb{F}}
\newcommand{\C}{\mathbb{C}}

\newcommand{\N}{\mathbb{N}}

\newcommand{\mV}{\mathcal{V}}
\newcommand{\mP}{\mathcal{P}}

\newcommand{\mC}{\mathcal{C}}
\newcommand{\mA}{\mathcal{A}}

\newcommand{\odelta}{\,\overline{\delta}\,}
\newcommand{\ox}{\,\overline{x}\,}

\newcommand{\ODelta}{\,\overline{\Delta}\,}

\newcommand{\oE}{\,\overline{E}\,}

\newcommand{\oDP}{\overline{\DP}\,}

\newcommand{\AP}{\mA_\mP}

\newcommand{\DP}{\mathcal{D}_\mP}

\DeclareMathOperator{\res}{Resultant}

\title{Sufficient conditions for the surjectivity of radical curve parametrizations}

\author{Jorge Caravantes$^1$,  J. Rafael Sendra$^2$,  David Sevilla$^3$ and Carlos Villarino$^4$ \\
$^1$ $^2$ $^4$ Universidad de Alcal\'a\\ Dpto. de F\'isica y Matem\'aticas \\ 28871 Alcal\'a de Henares (Madrid), Spain \\ 
$^3$  Centro U. de M\'erida\\ Universidad de Extremadura \\ Av. Santa Teresa de Jornet 38, 06800 M\'erida (Badajoz), Spain\\ 
email: {\tt$^1$jorge.caravantes@uah.es, $^2$rafael.sendra@uah.es,}\\ {\tt $^3$sevillad@unex.es, $^4$carlos.villarino@uah.es}
}

\begin{document}

{\color{red} The journal version of this paper appears in

\noindent
\textit{
Jorge Caravantes, J. Rafael Sendra, David Sevilla, Carlos Villarino,
Sufficient conditions for the surjectivity of radical curve parametrizations,
Journal of Algebra,
Volume 640,
2024,
Pages 129-146,
ISSN 0021-8693,
\url{https://doi.org/10.1016/j.jalgebra.2023.11.004}.
}

\vskip 20pt

\noindent
licensed under a Creative Commons Attribution 4.0 International License, \url{https://creativecommons.org/licenses/by/4.0/}}

%
%
%
%

\maketitle

\begin{abstract}
In this paper,  we introduce the notion of surjective radical parametrization and we prove sufficient conditions for a radical curve parametrization to be surjective. 
\end{abstract}

\noindent \textbf{Keywords.} Algebraic curve, radical parametrization, surjective parametrization, missing points.

\section{Introduction}
Let $\mV$ be a curve or an algebraic surface over $\C$ for which a parametrization $\mP$ is available. This parametrization can be visualized as a map from a subset of the parameter space on $\mV$. It therefore makes sense to talk about the surjectivity of $\mP$.

If the parametrization is not surjective, when working with it instead of the implicit equations of $\mV$, information about the part of $\mV$ not covered by $\mP$ is lost. This fact may have undesirable consequences in some applications since the information sought in $\mV$ may lie out of the image of $\mP$; in the introduction of \cite{SendraSevillaVillarino2017Ruled} this phenomenon is illustrated with some examples. Thus, the need to study the surjectivity of parametrizations arises. 

If $\mV$ is a curve and $\mP$ is a rational parametrization, the problem is essentially solved in \cite{Sendra} (see \cite{AR} and \cite{bajaj} for different approaches) where, in addition to giving sufficient conditions to guarantee surjectivity, it is shown that, at most, only one point can be missed via $\mP$, and it is shown how to detect it. In addition, it is also proved that every rational curve can be surjectively parametrized.

When $\mV$ is a surface, even if $\mP$ is a rational parametrization, the problem is considerably more complicated. In fact, in \cite{CSSV} it is shown that there are rational surfaces that do not admit surjective rational parametrizations. One way to overcome this difficulty is to look for finite families of parametrizations such that the union of their images covers $\mV$ (see \cite{Gao}, \cite{SendraSevillaVillarino2015}, \cite{SendraSevillaVillarino2017Ruled}). On the other hand, in \cite{SendraSevillaVillarino2015RICAM} particular families of surfaces are presented that do admit surjective rational parametrizations.

In this paper, we apparently reduce the complexity of the problem by working with curves instead of surfaces but increase the difficulty by allowing the parametrization to be radical; a radical parametrization is, intuitively, a tuple of rational functions whose numerators and denominators are nested expressions of radicals of polynomials with Jacobian rank equal to one, see \cite{SendraSevillaVillarino2017} for more details. Therefore, this work should be considered as a first step towards the general study of surjectivity in the radical case, leaving open problems for future work such as obtaining, if feasible, surjective radical parametrizations. This article focuses mainly on obtaining sufficient conditions to guarantee surjectivity. 

The first obstacle we face is to adequately define what we mean by surjectivity in the radical case. Let us see the problem in a trivial example. The rational parametrization $\left(\frac{2t}{t^2+1},\frac{t^2-1}{t^2+1}\right)$ of the unit circle covers all the curve except the north pole point $(0,1)$. However, if we consider the radical parametrization $(t,\sqrt{1-t^2})$ of the same curve, taking the square root as a function (i.e. just choosing one of the two possibilities for every radicand), we can only aspire to cover one half of the curve. Our approach is to refer to the union of the images of the given parametrization and all its conjugates. Thus, in the example above, we would be simultaneously talking about $(t,\pm \sqrt{1-t^2})$ and, now, the union of the images is the entire circle; see Example \ref{ex-circle} and Remark \ref{rem-fermat}.

Since, in the case of rational curve parametrizations, the only point of the curve that can be lost is the image of infinity, a sufficient condition for surjectivity is that at least one of the rational functions of the parametrization has a numerator with a degree strictly bigger than the degree of the denominator. With the intention of imitating this result in the radical case, we introduce weights to adequately define the degree of polynomials with nested radicals, see Def. \ref{def:grado}. However, this does not solve the problem completely (see Example \ref{ej:eje}). The reason is that indeterminations of type 0/0 can appear. To overcome this difficulty we introduce the notion of guilty polynomial (see Def. \ref{def:guilty}), and the weaker notion of suspicious polynomial (see Def. \ref{def-sup}), to finally give a sufficient condition of radical surjectivity, see Theorem \ref{th:main} and Corollary \ref{cor-susp}.

The paper is structured as follows: in Section \ref{sec:previos} we recall some notions and results of radical parametrizations, we introduce the notion of surjectivity for a radical parametrization, and we outline our strategy. In Section \ref{sec-suff} we introduce the notion of guiltiness and we prove the main Theorem \ref{th:main}. Some consequences are derived there. In Section \ref{sec-hypo} we discuss the hypotheses of the main theorem and present an alternative condition to guiltiness, namely being suspicious, that is more convenient from a computational point of view but in principle more restrictive (see Theorem \ref{thm-sus-guilty}). In Section \ref{sec-miss} we conclude the article with bounds on the number of missing points in the case when some of our conditions for surjectivity do not hold.

\section{Notation and preliminaries}\label{sec:previos}

A radical parametrization is a {tuple $(x_1,\ldots,x_n)$} of elements of the last field of a radical tower $\F_0\subseteq\cdots\subseteq\F_m$ where $\F_0=\C(t)$ and $\F_i=\F_{i-1}(\delta_i)$ such that $\delta_i^{e_i}\in\F_{i-1}$ (see more details in \cite{SendraSevillaVillarino2017}). We will work with the polynomial ring $\C[t,\Delta_1,\ldots,\Delta_m]$ where weights on the variables will be introduced later. In this ring lie the defining polynomials of the $\delta_i$, namely $E_i:=\Delta_i^{e_i}-g_i(t,\Delta_1,\ldots,\Delta_{i-1})$ where $\deg_{\Delta_j}(g_i)<e_j$ for all $j=1,...,i-1$. We will denote the $\Delta_i$ collectively as $\ODelta$, and similarly for other indexed names.

\begin{remark}\label{rem:estructura}
Any rational function in $\C(t,\odelta)$ can be written as a polynomial in $\C(t)[\odelta]$.
\end{remark}

Recall from \cite{SendraSevillaVillarino2017} that a radical parametrization of a curve, among other stuff, incorporates a pair $(\mP, \oE)$, where
\begin{equation}\label{eq:P}
{
\mP(t,\ODelta)=\left(\displaystyle\frac{p_1(t,\ODelta)}{q_1(t,\ODelta)},\ldots,\frac{p_n(t,\ODelta)}{q_n(t,\ODelta)}\right)}
\end{equation}
is an element of {$\C(t,\ODelta)^n$}, and $\oE=(E_1,\ldots,E_m)$. We will use throughout this paper the following definition of surjectivity:

\begin{definition}\label{def:sup}
We say that the parametrization is \textbf{surjective} or \textbf{normal} when the set
\[X(\mP):=\left\{{\overline{x}\in\C^n} \left|\ \exists\, (t,\ODelta)\in\C^{m+1}\mbox{ \rm s.t. }
\begin{array}{l}
E_i(t,\ODelta)=0,\ \forall\, i=1,\ldots,m; \cr
\mP(t,\ODelta)\mbox{ \rm is well defined; and}\cr
\mP(t,\ODelta)={\overline{x}}
\end{array}
\right.\right\}\]
is Zariski closed in $\C^n$. 
\end{definition}

Let us call $\DP$ the algebraic subset of $\C^{m+n+2}$ of points $(t,\ODelta,\overline{x},z)$ satisfying the equations $E_i(t,\ODelta)=0$, $i=1,\ldots,m$; 
{$p_i(t,\ODelta)-x_i \,q_i(t,\ODelta)=0$; $i=1,\ldots,n$; and $z\cdot\mathrm{lcm}(q_1(t,\ODelta),\ldots,q_n(t,\ODelta))-1=0$.} 
We have then the projection
\begin{equation}\label{eq-DP}
\begin{array}{ccc}
\DP
&{\longrightarrow}&\overline{X(\mP)} {\subset \C^n}\\
{(t,\ODelta,\overline{x},z)}
&\mapsto&{\overline{x}}
\end{array}
\end{equation}
where $\overline{X(\mP)}$ denotes the Zariski closure of $X(\mP)$. Note that every constructible set involved is unidimensional (see \cite{SendraSevillaVillarino2017}). Then, the set of missing points $\overline{X(\mP)}\setminus{X(\mP)}$ is finite. Observe that the radical variety defined in \cite{SendraSevillaVillarino2017}, in our case, radical curve, is an irreducible component of $\overline{X(\mP)}$, so, when irreducible, $\overline{X(\mP)}$ is the whole radical variety.

\begin{remark}
The definition of surjective in Def. \ref{def:sup} corresponds to the usual notion of surjective map in the following way. For each combination of possible branches in \eqref{eq:P}, we can introduce a map from a subset of $\C$ to $\C^n$. Surjectivity in Def. \ref{def:sup} requires that $\overline{X(\mP)}=X(\mP)$ (see \eqref{eq-DP}), and hence that the union of the images of all these maps is $\overline{X(\mP)}$. 
\end{remark}

\begin{example}\label{ex-circle}
Let $\mP(t,\Delta)=(t, \Delta)$ where $E:=\Delta^2-(1-t^2)$. Then, $\DP$ is 
\[ \DP=\left\{ (t,\Delta,x,y,z) \in \C^5 \,\left| \, \begin{array}{r} x-t=0 \\ y-\Delta=0 \\ z=1 \\ \Delta^2-(1-t^2)=0 \end{array} \right. \right\}. \]
By the Closure Theorem (see \cite[Ch. 3.2, Th. 3 or Ch. 4.4, Th. 3]{CoxLittleOshea2015a}) we have that $\overline{X(\mP)}$ is the curve $\mathcal{C}:=\{(x,y)\in \C^2\,|\, x^2+y^2=1\}$. Let us now analyze $X(\mP)$, that is 

\[ X(\mP):=\left\{ (x,y) \in\C^2 \left|\ \exists\, (t,\Delta)\in\C^{2} \quad\mathrm{s.t. }
\begin{array}{l}
\Delta^2=1-t^2 \cr
(t,\Delta)=(x,y).
\end{array}
\right.\right\} \]

First, we observe that $X(\mP)\subset \mathcal{C}$ since $t^2+\Delta^2=1$. On the other hand, if $(a,b)\in \mathcal{C}$, taking $t=a$, and $\Delta$ as one of the roots of $\sqrt{1-a^2}$, we get that that $(a\,b)\in \DP$, and hence $X(\mP)=\overline{X(\mP)}$. So $X(\mP)$ is Zariski closed and hence, by Def. \ref{def:sup}, $\mP$ is surjective.

On the hand $\mP$ provides two map, namely, $(t,\pm \sqrt{1-t^2})$, and the union of their images is $\mathcal{C}$.
\end{example}

Recall the notation $V(I)$ for the algebraic subset of the affine space defined by an ideal $I$ and $I(X)$ for the ideal of polynomials vanishing at all the points of a subset $X$ of the affine space. Again by the Closure Theorem, we have that $\overline{X(\mP)}=V(I(\DP)\cap\C[{\overline{x}}])$. 
Observe that $X(\mP)$ is the image of $\DP$ by the projection $(t,\ODelta,{\overline{x}},z)\mapsto{\overline{x}}$. This means that $I(X(\mP))=I(\overline{X(\mP)})\supset I(\DP)\cap\C[{\overline{x}}]$. Therefore, surjectivity is equivalent to the image of $\DP$ by the projection being Zariski closed. Given this fact, our main tool will be the Extension Theorem:

\begin{theorem}\label{th:extension}[Extension Theorem, see \cite{CoxLittleOshea2015a} Ch. 3.1, Th. 3]
Let $I=\langle f_1,\ldots,f_s\rangle\subseteq\C[x_1,\ldots,x_n]$ and let $\widetilde{I}=I\cap\C[x_2,\ldots,x_n]$. For each $1\leq i\leq s$, write
\[
f_i = c_i(x_2,\ldots,x_n) x_1^{N_i} + \mbox{ terms in which $x_1$ has degree $<N_i$},
\]
where $N_i\geq0$ and $c_i\neq0$. Suppose that we have a partial solution $(a_2,\ldots,a_n)\in V(\widetilde{I})$. If there exists $c_i$ which is nonzero at $(a_2,\ldots,a_n)$ then there exists $a_1\in\C$ such that $(a_1,a_2,\ldots,a_n)\in V(I)$.
\end{theorem}

Theorem \ref{th:extension} can be used to lift any partial solution ${\overline{x}}$ of $I(\DP)\cap\C[{\overline{x}}]$ to $\DP$, proving surjectivity.

In the case of rational curve parametrizations, 
to check whether the rational parametrization is surjective it suffices to check that for one of the components the degree of the numerator is larger than its denominator's. In order to replicate this result for radical parametrizations, we will provide $\C[t,\ODelta]$ with a structure of graded ring with rational weights for the variables.

\begin{definition}\label{def:grado}
First we define $\deg(t)=1$. Then we define recursively the \textbf{degree} of $\Delta_i$ as the degree of $g_i$ divided by $e_i$, with the classical definition of degree for a polynomial on several weighted variables; {recall that $\Delta_{i}^{e_i}=g_i$}.
\end{definition}

\begin{example}\label{ej:eje}
This example illustrates how the condition on the degrees that works in the rational case is not sufficient in the radical case. Consider the parametrization ${\mP(t)}=(0,t-\sqrt{t^2-1})$ of the vertical axis {$x=0$ in $\C^2$}. If we take $\Delta_1$ to be a square root of $g_1(t)=t^2-1$, the degree of the polynomial corresponding to the $y$ coordinate, $y(t,\Delta_1)=t\pm\Delta_1$, would be 1 in any case. Since the denominator of $y$ is constant, the desired condition on the degrees is true. However, the origin does not correspond to any value of $t\in\C$ with any of its two possible square roots.
\end{example}

We need to exclude some expressions in order to find easy to check sufficient conditions of surjectivity. In order to apply the Extension Theorem to polynomials that will arise from eliminating the $\Delta_i$ by multiplying conjugates, we need conditions to ensure that we can control the degrees of those products. 
Note that one can eliminate the $\Delta_i$ with the set $\{E_1,\ldots,E_m\}$, which is a triangular Gr\"obner basis with respect to the lexicographical ordering $\Delta_m>\cdots>\Delta_1>t$.

\begin{definition}\label{def-NF}
For every $f\in\C[t,\ODelta]$, we denote as $N(f)$ the \textbf{normal form} {of $f$} with respect to $\{E_1,\ldots,E_m\}$. Note that the normal form is just the polynomial obtained by repeatedly substituting every instance of $\Delta_i^{e_i}$ by $g_i$.

If $f\in \C[t]$, observe that $N(f)=f$.
\end{definition}

\begin{example}[continuation of Example \ref{ej:eje}]\label{ej:eje2}
The problem in Example \ref{ej:eje} is that $t=\infty$ corresponds to an affine point in the parametrization. There we have two leading monomials, $t$ and $\sqrt{t^2-1}$, that in a way cancel each other. In fact, to reach such cancellation with radicals, one usually multiplies by the conjugate. For ${f(t,\Delta_1)}=t-\Delta_1$ and its conjugate $f(t,-\Delta_1)=t+\Delta_1$, the normal form of their product (recall $E_1=\Delta_1^2-(t^2-1)$) is $N(t^2-\Delta_1^2) = \mathrm{rem}(t^2-\Delta_1^2,E_1) = 1$. {In terms of radicals, we can write} $(t-\sqrt{t^2-1})(t+\sqrt{t^2-1}) = 1$. Note that we lose degree when, after multiplying by all conjugates of $f$, we apply the equality $E_1=0$ to subtitute $\Delta_1^2$ in terms of $t$.
\end{example}

%
%

\section{Sufficient condition for surjectivity}\label{sec-suff}

Let us denote by $\AP$ the Zariski closure of the projection of $\DP$ via the map $(t,\ODelta,\overline{x},z)\mapsto (t,\ODelta,\overline{x})$. We will apply repeatedly the Extension Theorem (Theorem \ref{th:extension}) to lift a point ${\overline{x}}$ in $\overline{X(\mP)}$ (i.e. a solution of $I(\DP)\cap\C[{\overline{x}}]=I(\AP)\cap\C[{\overline{x}}]$), first to $\AP$, and later to $\DP$. In the first step, we will control the coefficients of the generator of the ideal $I(\AP)\cap\C[{x_1},t]$ (and similarly for {$x_i$ with $i>1$}) so that for every point of $\overline{X(\mP)}$ there exists a corresponding value of $t$ (see Theorem \ref{th:main}). To this end we will consider a product involving the conjugates of the numerator of the first component of $\mP$ whose degree should be preserved upon simplification of roots. This is the motivation for the next definition. From here onwards we will denote the imaginary unit with an upright $\mathrm{i}$ as opposed to the italic $i$ for a variable.

\begin{definition}\label{def:Rf}
For $k=1,\ldots,m$, denote $\gamma_k=\exp(2\pi\mathrm{i}/e_k)$. Let $f\in\C[t,\ODelta]$. We define $f_m=f(t,\Delta_1,\ldots,\Delta_m)$ and recursively for $k=m,\ldots,1$,
\begin{equation}\label{eq:Fk}
 F_k=\prod_{j=1}^{e_k} f_k(t,\Delta_1,\ldots,\Delta_k\cdot\gamma_k^j),
\end{equation}
\begin{equation}\label{eq:fk}
 \quad f_{k-1} = \mathrm{rem}(F_k,E_k)\in \C[t,\Delta_1,\ldots,\Delta_{k-2}][\Delta_{k-1}] \mbox{ (see Lemma \ref{lema:unavariablemenos})}.
\end{equation}
Then we define the \textbf{normalized reminder} of $f$ as $R(f)=f_0\in\C[t]$. 

If $f\in \C[t]$, observe that $R(f)=f$.
\end{definition}

\begin{lemma}\label{lema:unavariablemenos}
With the notation of Definition \ref{def:Rf}, $f_{k-1}\in\C[t,\Delta_1,\ldots,\Delta_{k-1}]$.
\end{lemma}

\begin{proof}
By construction,
\[
F_k(t,\Delta_1,\ldots,\Delta_k) = F_k(t,\Delta_1,\ldots,\gamma_k\Delta_k).
\]
Writing $F_k$ as a polynomial in $\Delta_k$, namely $F_k=\sum_h a_h\cdot\Delta_k^h$, we have $a_h=a_h\cdot\gamma_k^h$. Therefore, if $e_k$ does not divide $h$, $\gamma_k^h\neq1$ so $a_h=0$. It follows that $F_k$ is a polynomial in $t,\Delta_1,\ldots,\Delta_{k-1},\Delta_k^{e_k}$:
\[
F_k=\sum_{l}a_{le_k}\Delta_l^{le_k},
\]
so its remainder when dividing by $E_k$ with respect to $\Delta_k$ is:
\[
f_{k-1}= \sum_{l}a_{le_k}g_k(t,\Delta_1,\ldots,\Delta_{k-1})^{l}.
\]
\end{proof}

\begin{remark}\label{rem:R(f)}\
\begin{enumerate}
\item $f_{k-1}$ is the resultant of $f_k$ and $E_k$ with respect to $\Delta_k$, since the resultant is the product of the evaluations of $f_k$ in the roots of $E_k$. In other words, $R(f)$ is obtained by nested resultants.
\item We defined $f_{k-1}$ as a remainder by a polynomial but using the normal form instead one obtains the same $R(f)$.
\item In general, $R(f)$ is a power of the norm of $f(t,\odelta)$ as an algebraic element over $\C(t)$.
\item If the radicals are not nested, i.e. every $g_k\in\C[t]$, we can substitute the recursive definition by the product
\[
R(f) = N\left(\prod_{i_1,\ldots,i_m} f(t,\gamma_1^{i_1}\Delta_1,\ldots\gamma_m^{i_m}\Delta_m)\right).
\]
This is not true in general, as the next example illustrates.
\end{enumerate}
\end{remark}

\begin{example}
Consider $E_1=\Delta_1^2-t$, $E_2=\Delta_2^2-\Delta_1-1$ and $f=\Delta_1\Delta_2+t$. If we compute the whole product instead of taking remainders after conjugating each variable separately, we obtain
\[
N\left(\prod f(t,\pm\Delta_1,\pm\Delta_2)\right) = (-2t^3+2t^2)\Delta_1+t^4-t^3+t^2.
\]
Following Definition \ref{def:Rf}, we get $R(f)=t^4-3t^3+t^2$.
\end{example}

Since conjugation does not change the degree, from \eqref{eq:Fk}, we get $\deg(F_k)=e_k\cdot\deg(f_k)$; on the other hand $\deg(\Delta_k)=\deg(g_k)/e_k$ by Definition \ref{def:grado}, so reducing $F_k$ by $E_k$ (i.e. substituting $\Delta_k^{e_k}=g_k$ in $F_k$) will not change the degree in the general case. Therefore, the degree of $R(f)$ is, at most, $\deg(f)\cdot e_1\cdots e_m$, with equality as the expected case. However, Example \ref{ej:eje2} shows that this is not always true, due to cancellations upon reduction. Now we introduce a definition for the exceptions.

\begin{definition}\label{def:guilty}
Let $f\in\C[t,\ODelta]$. We define $f$ to be \textbf{guilty} when
\[
 \deg(f)\cdot e_1\cdots e_m > {\deg(R(f))}.
\]
\end{definition}

\begin{remark}\label{rem-rat-guilty}
By definition, no polynomial in $\C[t]$ is guilty.
\end{remark}

\begin{example}[continuation of Example \ref{ej:eje2}]
We have
\[
 f_1 = f = t-\Delta_1, \ F_1 = (t-\Delta_1)(t+\Delta_1), \ f_0 = R(f) = 1.
\]
Since $\deg(f)=1$ and $e_1=2$, $f$ is guilty. If we had started instead with $y=t-\sqrt{t-1}$, we would have had the same $f=t-\Delta_1$ but now with $E_1=\Delta_1^2-(t-1)$, so that $R(f) = t^2-t+1$ and $f$ would not be guilty; one can also check that the parametrization would be surjective.
\end{example}

With the previous definitions we pretend to control the ``image of infinity'' as in rational parametrizations. However, there is an extra issue that may appear in with radical, which is the posibility of indeterminations of type $0/0$. These happen for those values of $t, \ODelta$ satisfying the equations $\oE$, $p_i$ and $q_i$, where $i$ indicates the component of $\mP$ where the indetermination occurs. 
Now we can state the main theorem.


\begin{theorem}\label{th:main}
Let $\mP$ be as in \eqref{eq:P}, with $N(p_i)=p_i$ and $N(q_i)=q_i$ for {$i=1,\ldots,n$}, and $\oE$ be the sequence of defining polynomials of the radical expressions.
Suppose that:
\begin{enumerate}
\item there exists $i$ such that $p_i$ is not guilty and $\deg p_i>\deg q_i$;
\item for {all} $i$, the ideal $I_i(\mP):=\langle \oE,p_i,q_i\rangle$ is the whole $\C[t,\ODelta]$.
\end{enumerate}
Then the parametrization is surjective.
\end{theorem}

\begin{proof}
Recall that $\AP$ is the Zariski closure of the projection $(t,\ODelta,\overline{x},z) \mapsto (t,\ODelta,\overline{x})$ of $\DP$. 
We will repeatedly apply the Extension Theorem (Theorem \ref{th:extension}) to lift a point ${\overline{x}}\in\overline{X(\mP)}=V(I(\DP)\cap\C[{\overline{x}}])$, firstly, to $\AP$, and, later, to $\DP$.

Without loss of generality, assume that hypothesis 1 holds for $i=1$. Let $F_1({x_1},t)$ be the generator of the ideal $I(\AP)\cap\C[{x_1},t]$. This makes sense because each component of $\AP$ has dimension one (see \cite[Th. 3.11]{SendraSevillaVillarino2017}) and its projection to the variables $({x_1},t)$ is finite (in \cite[Proof of Th. 3.11]{SendraSevillaVillarino2017}, the {$(t,\ODelta,\overline{x},z) \mapsto (t,\ODelta,\overline{x})$ of $\DP$} is finite and, since the dimension is one, the missing points of the $(t,\ODelta,\overline{x},z) \mapsto (t,\ODelta,\overline{x})$ are a finite set). Let
\[
G_1({x_1},t) = R({x_1}q_1-p_1). 
\]
By construction $G_1\in I(\AP)$ so $F_1$ divides $G_1$. In Lemma \ref{lemma:G1} below we show that, if we consider $G_1$ as a polynomial in {$x_1$}, its constant coefficient has greater degree in $t$ than the other coefficients. Therefore the leading coefficient of $G_1$ with respect to $t$ is constant, and the same thing happens to $F_1$. Then, every point of $\overline{X(\mP)}$ can be lifted to a value of $t$ by Theorem \ref{th:extension}. But since each polynomial $E_i=\Delta_i^{e_i}-g_i$ is monic with respect to $\Delta_i$ and they are in $I(\AP)$, every point can be lifted to the $\ODelta$ as well.

Finally, consider a point ${\overline{x}_0}$ in the curve and all its lifts to the values $t_0,\ODelta_0$. We need to lift to the variable $z$ via the condition $z\cdot \mathrm{lcm}({q_1,\ldots,q_n})=1$.

Since $t_0,\ODelta_0,{\overline{x}_0}$ satisfy the equations {$q_ix_i-p_i=0$ for all $i\in\{1,\ldots,n\}$} and the $\oE$ equations, if any $q_i(t_0,\ODelta_0)=0$ we would also have $p_i(t_0,\ODelta_0)=0$ contradicting hypothesis 2.
\end{proof}

\begin{remark}\label{rem-rat} Let $\mP(t)=(p_1(t)/q_1(t),\ldots,p_n(t)/q_n(t))\in \C(t)^n$ be a rational parametrization in reduced form; that is, $\gcd(p_i,q_i)=1$ for all $i\in \{1,\ldots,n\}$. Let us assume that there exists $i_0\in \{1,\ldots,n\}$ such that $\deg(p_{i_0})> \deg(q_{i_0})$. Then $\mP(t)$ satisfies the hypotheses of Theorem \ref{th:main}. 

\vspace*{1mm}

\noindent 
Indeed, by Def. \ref{def-NF}, it holds that $N(p_i)=p_i$ and $N(q_i)=q_i$ for all $i$. By 
Remark \ref{rem-rat-guilty}, $p_{i_0}$ is not guilty and $p_{i_0}/q_{i_0}$ satisfies hypothesis (1). Concerning hypothesis (2) one has that, since $\gcd(p_i,q_i)=1$, then 
$1\in I_i(\mP)=\langle p_i,q_1\rangle$. So $I_i(\mP)=\C[t,\ODelta]$.
\end{remark}

\begin{corollary}\label{cor-pol}
Let $\mP$ be as in \eqref{eq:P}, with all its components being reduced polynomials in $\C[t,\ODelta]$ (i.e. $N(p_i)=p_i$ and $q_i=1$ for {$i=1,\ldots,n$}). If hypothesis (1) in Theorem \ref{th:main} is satisfied, then $\mP$ is surjective.
\end{corollary}

\begin{proof}
Let us see that the second hypothesis of Theorem \ref{th:main} is satisfies. Let $p_i\in\C[t,\ODelta]$ be the $i$-th component of $\mP$, then $I_i(\mP)=\langle\overline{E},p_i, 1\rangle=\langle1\rangle=\C[t,\ODelta]$. 
\end{proof}

\begin{corollary}\label{cor-3}
Let $\mP$ be such that some of its components are rational and one of its rational components satisfies the degree condition. If every non-rational components $p_i/q_i$ is polynomial in $\C[t,\ODelta]$, and $N(p_i)=p_i$, then $\mP$ is surjective. 
\end{corollary}

\begin{proof}
Taking into account Def. \ref{def-NF}, and the hypothesis on the non-rational components of $\mP$, we have that $N(p_i)=p_i$ and $N(q_i)=q_i$ for all $i$. Also, reasoning as in Remark \ref{rem-rat}, we know that the hypothesis (1) of Theorem \ref{th:main} holds. Now the result follows from Corollary \ref{cor-pol}.
\end{proof}

\begin{remark}
The curve of $\C^n$ defined by the complex polynomials $$\{ y_{i}^{n_i}-g_i(x)\}_{i=1,\ldots,n-1},$$ and $n_i\in \N$, can be surjectively parametrized as
\[
\mP=(t,\Delta_1,\ldots,\Delta_{n-1})
\]
where $E_i:=\Delta_{i}^{n_{i}}-g_i(t)$.

\vspace*{1mm}

\noindent
We observe that $N(\Delta_i)=\Delta_i$. Thus, the claim follows from Corollary \ref{cor-3}.
\end{remark}

\begin{remark}\label{rem-fermat}
As particular examples of the previous remark, we have the plane curves $x^n+y^n=a$, with $a\in \C$.
\end{remark}

\begin{lemma}\label{lemma:G1}
With the notation and hypotheses of Theorem \ref{th:main}, let $f=x q -p$. Then, abusing the notation of Definition \ref{def:Rf}, for every $k=0,\ldots,m$, the constant coefficient of $f_k$ with respect to $x$ has degree $e_{k+1}\cdots e_m\cdot\deg(p_1)$ which is strictly greater than that of the other coefficients. In particular, this applies to $G_1=R(f)=f_0$.
\end{lemma}

\begin{proof}
We apply descending induction on $k$; the case $k=m$ is trivial by hypothesis.

Assume it true for $k$, that is, if $f_k=a_rx^r+\cdots+a_0$ with $a_j\in\C[t,\Delta_1,\ldots,\Delta_k]$ then $\deg(a_0)=e_{k+1}\cdots e_m\cdot\deg(p_1)>\deg(a_j)$ for $j>0$. Let
\[
 F_k = \prod_{i=1}^{e_k}\left( a_r(t,\Delta_1,\ldots,\Delta_k\cdot\gamma_k^i)x^r+\cdots+a_0(t,\Delta_1,\ldots,\Delta_k\cdot\gamma_k^i)\right) = b_{re_k}x^{re_k}+\cdots+b_0
\]
Observe that $\deg(b_0)=\deg(a_0)\cdot e_k=e_k\cdots e_m\cdot\deg(p_1)$, the last equality by induction hypothesis.

For $j>0$, $b_j$ is a sum of products of exactly $e_k$ of the conjugated $a$'s, not all being $a_0$. Since the degree of $a_0$ is strictly greater than the degree of the others, all summands in $b_j$ have degree strictly less than $e_k\cdot\deg(a_0)=\deg(b_0)$.

Now define $\tilde{b}_j$ such that $\tilde{b}_j(t,\Delta_1,\ldots,\Delta_{k-1},\Delta_k^{e_k})=b_j$. Then
\[
 f_{k-1}=\tilde{b}_{re_k}(t,\Delta_1,\ldots,\Delta_{k-1},g_k)x^{re_k}+\cdots+\tilde{b}_0(t,\Delta_1,\ldots,\Delta_{k-1},g_k).
\]
The constant term of $f_{k-1}$ above coincides with the $k-1$ step in the computation of $R(p_1)$, possibly up to a sign. Since $p_1$ is not guilty, in each reduction step for $R(p_1)$ the degree does not drop. This shows that
\[
 \deg(\tilde{b}_0(t,\Delta_1,\ldots,\Delta_{k-1},g_k))=\deg(b_0)>\deg(b_j)\geq\deg(\tilde{b}_j(t,\Delta_1,\ldots,\Delta_{k-1},g_k)).
\] 
\end{proof}

\section{On the hypotheses for surjectivity}\label{sec-hypo}


Deciding whether a polynomial is guilty or not can be done by defining and checking a sort of leading coefficient for our polynomials. We will see this for $m=1$ (i.e. $\ODelta=(\Delta_1)$). The generalization to $m>1$ is natural in the unnested case (i.e. $g_i(t,\Delta_1,\ldots,\Delta_{i-1})\in\C[t]$). For the nested case, see Remark \ref{rm:caracterizacion_de_culpabilidad}. 

Let us denote $g(t)=g_1(t)=a_0+\cdots+a_kt^k$, $\Delta=\Delta_1$, $\delta=\delta_1$. Consider an already reduced polynomial 
\[
f(t,\Delta)=c_0(t)+c_1(t)\Delta+\cdots+c_{e-1}(t)\Delta^{e-1}\in \C[t,\Delta].
\]
Since we defined $\deg(\Delta)=k/e$, naming $\partial c_i:=\deg_t(c_i(t))$, we have that for $i\in\{0,\ldots,e-1\}$
\[
M:=\deg(f(t,\Delta))=\max\{\deg(c_i(t)\Delta^i)\}=\max\left\{\partial
c_i+\frac{k}{e}i\right\}.
\]
According to Remark \ref{rem:R(f)},
\[
R(f)=R(t)=\res_{\Delta}(f(t,\Delta),\Delta^e-g(t))=\prod_{r=0}^{e-1} f(t,\gamma^r\delta(t)),
\]
with $\gamma=\exp(2\pi\mathrm{i}/e)$. Now, define
\[
J=\{i\in\{0,\ldots,e-1\}\,\,/\,\,\partial c_i+\frac{k}{e}i=M\}
\]
and 
\[
f_l(\Delta)=\sum_{i\in J} c_{i \partial c_i}\Delta^i\in \C[\Delta],
\]
where $c_{i \partial c_i}$ is the leading coefficient of $c_i(t)=c_{i0}+c_{i1}t+\cdots+c_{i\partial c_i}t^{\partial
c_i}$. For any factor of $R(t)$
\[
f(t,\gamma^r\delta)=c_0(t)+c_1(t)\gamma^r\sqrt[e]{g(t)}+\cdots+c_{e-1}(t)\gamma^{r(e-1)}\sqrt[e]{(g(t))^{e-1}},
\]
its coefficient of degree $M$ is 
\[
\sum_{i\in J}\,c_{i\partial c_i}\gamma^{ri}(a_k)^{\frac{i}{e}}
\]
(recall $a_k$ is the leading coefficient of $g(t)$). With this notation, we state the following result.

\begin{lemma}\label{lma:culpa_facil}
With the above notation, the following are equivalent:
\begin{enumerate}
\item $f(t,\Delta)$ is not guilty.
\item For all $r\in \{0,\ldots,e-1\}$, $\displaystyle{\sum_{i\in
J}\,c_{i\partial c_i}\gamma^{ri}a_k^{i/e}\neq 0}$.
\item $\res_{\Delta}(f_{l}(\Delta),\Delta^e-a_k)\neq 0$.
\end{enumerate}
\end{lemma}

\begin{proof}
Let $L=\displaystyle{\prod_{r=0}^{e-1}\sum_{i\in J}\,c_{i\partial
c_i}\gamma^{ri}a_k^{i/e}}=\res_{\Delta}(f_{l}(\Delta),\Delta^e-a_k))$. The right equality proves $(2)\Leftrightarrow(3)$.

We observe that, as seen in the proof of Lemma \ref{lema:unavariablemenos}, $R(f)$ is the substitution $\Delta^e=g(t)$ in the product of all $f(t,\gamma^r\Delta)$, which depends on $t$ and $\Delta^e$. The highest degree homogeneous component (with our definition of degree) of such polynomial is the product of the conjugates of the $M-th$ homogeneous component of $f(t,\Delta)$:
\[
A(t,\Delta^e)=\prod_{r=0}^{e-1}\sum_{i\in J} c_{i \partial c_i}t^{\partial c_i}\gamma^{ri}\Delta^i.
\]
On the other side, and reasoning likewise, $\res_{\Delta}(f_{l}(\Delta),\Delta^e-a_k))$ is the substitution $\Delta^e=a_k$ in the product of all $f_l(\gamma^r\Delta)$:
\[
B(\Delta^e)=\prod_{r=0}^{e-1}\sum_{i\in J}\,c_{i\partial c_i}\gamma^{ri}\Delta^i.
\]

Since $B(a_k)=L$,
\[
R(f)=\res_\Delta(f(t,\Delta),\Delta^e-g(t))=Lt^{ M\cdot e}+\mbox{lower degree terms},
\]
which means that $(1)$ is equivalent to $(3)$.
\end{proof}

\begin{remark}\label{rm:caracterizacion_de_culpabilidad}
Lemma \ref{lma:culpa_facil} should be easily extendable to the case with not nested radicals and, with some more work, maybe, to the general case. However the notation needed for the definition of $f_l(\ODelta)$, mainly in the general case, would be cumbersome.
\end{remark}

The following examples show that the hypotheses in Theorem \ref{th:main} are sufficient but not necessary.

\begin{example}
On the condition of guiltiness, consider the radical parametrization $x=y=t(\sqrt{t}-\sqrt{t+1})$. It can be checked that it is surjective, although it is defined by $p_1=t(\Delta_1-\Delta_2)$ with $\Delta_1^2=t$, $\Delta_2^2=t+1$ which is guilty. More in detail, the polynomial that relates $t$ and $x$ (what we called $G_1$ in the proof of Theorem \ref{th:main}) is
\[
t^4-4x^2t^3-2x^2t^2+x^4,
\]
which is monic in $t$ of degree 4, while from the theorem the expected degree would be $\frac32\cdot4=6$; the degree drop is due to the reduction occurring by multiplication of the four conjugates. However, since it is a monic polynomial, we can apply Theorem \ref{th:extension} for any value of $x$.
\end{example}

\begin{example}
On the degree condition, the Bernoulli lemniscate $(x^2+y^2)^2=x^2-y^2$ can be parametrized by 
\[
 x = \frac{t+t^3}{1+t^4}, \quad y = \frac{t-t^3}{1+t^4}
\]
which is surjective although the degree condition is not fulfilled. See \cite{Sendra} for more details.
\end{example}

\begin{example}
Consider the parametrization $\left(\frac{t(\sqrt{t}-1)}{t-1},\ \frac{t^2+1}{t-1}\right)$. The pair $t=1,\Delta_1=1$ is a zero of the ideal $I_1(\mP)=\langle\Delta_1^2-t,t(\Delta_1-1),t-1\rangle$, so hypothesis 2 is not satisfied. However, this is not a problem with respect to surjectivity. Indeed, since hypothesis 1 is satisfied, every point of $\overline{X(\mP)}$ can be lifted to $t,\ODelta$; after this, even though we do not have hypothesis 2 to guarantee the lifting to $z$ for $t=1$, there is no affine point in that situation because the second component becomes infinity for that parameter value.
\end{example}

\begin{remark}
The solutions of the ideal $I_i(\mP)$ are instances of indeterminations of the type $0/0$ for the $i$-th component. Since the subideal generated by $\oE$ has dimension 1 in $\C[t,\ODelta]$, three things can happen to the solution set of $I_i(\mP)$:
\begin{itemize}
\item it is empty, which corresponds by Hilbert's Nullstellensatz to the part of hypothesis 2 regarding the $i$-th component;
\item it is finite;
\item it is unidimensional. This indicates that both $p_i$ and $q_i$ become identically zero for some component of the unidimensional $V(\oE)$. Since $\oE$ is a Gr\"obner basis by itself and both $p_i$ and $q_i$ are assumed to be in normal form with respect to $\oE$, this cannot happen when $V(\oE)$ is irreducible (i.e. when the subideal generated by $\oE$ is primary).
\end{itemize}
\end{remark}

The rest of the section discusses more restrictive conditions than guiltiness that are, on the other hand, easier to check. That is, we can substitute hypothesis (2) of Theorem \ref{th:main} with something else that still provides surjectivity.


\begin{proposition}
With the previous notations, if $f\in\C[t,\ODelta]$ and $(t_0,\ODelta_0)\in\C^{m+1}$ satisfy $f(t_0,\ODelta_0)=0$ and $E_i(t_0,\ODelta_0)=0$ for every $i$, then $R(f)(t_0)=0$.
\end{proposition}

\begin{proof}
By hypothesis $f_m=f$ vanishes at the point. Since each $F_k$ is a multiple of $f_k$ and each $f_{k-1}$ is the remainder of $F_k$ by $E_k$, it follows that all these polynomials vanish as well. But $R(f)=f_0$.
\end{proof}

As a consequence, we have the following condition, which is computationally more convenient, for instance with resultants.
\begin{corollary}
If $\gcd(R(q_i),R(p_i))=1$ for all $i=1,\ldots, n$, then hypothesis 2 of Theorem \ref{th:main} holds.
\end{corollary}

Next we offer another approach. Since checking guiltiness involves considering potentially too many conjugates and reducing, the following definition is convenient in a computational sense.

\begin{definition}\label{def-sup}
We say, again recursively, that a polynomial $f\in\C[t,\ODelta]$ is \textbf{suspicious} when either of these occurs:
\begin{enumerate}
\item there are at least two terms of highest degree; or
\item for some $i$, $\Delta_i$ appears in the (only) leading term and $g_i(t,\Delta_1,...,\Delta_{i-1})$ is suspicious.
\end{enumerate}
Note that, when $g_i$ is suspicious, the polynomial $\Delta_i$ is also suspicious.
\end{definition}

\begin{example}\label{ej:sospechoso}
Consider the expression 
\[
\sqrt{t^2-1}\sqrt{t-\sqrt{t^2-1}}+3.
\]
Naming $\delta_1=\sqrt{t^2-1}$ and $\delta_2=\sqrt{t-\delta_1}$, we have that $f(t,\Delta_1,\Delta_2)=\Delta_1\Delta_2+3$, which has degree $3/2$. It is a suspicious polynomial, since $g_2(t,\Delta_1)=t-\Delta_1$ is clearly suspicious and $\Delta_2$ appears in the leading monomial.

However, $f(t,\Delta_1,\Delta_2)+t^2$ is not suspicious because its degree is 2 and the highest degree homogeneous component only has the term $t^2$.
%
\end{example}

\begin{remark}
If the radicals are not nested, then the $g_i$ cannot be suspicious. In particular this happens if the radical parametrization is defined with only one root.
\end{remark}

\begin{theorem}\label{thm-sus-guilty}
If a polynomial is not suspicious then it is not guilty.
\end{theorem}

\begin{proof}
Let $f$ be not suspicious. With the notations of Definition \ref{def:Rf}, we will prove the following by descending induction: every $f_k$ is not suspicious and $\deg(f_{k-1})=e_k\cdot\deg(f_k)$. The case $k=m$ is trivial. Now we suppose it true for $k$ and prove it for $k-1$.

For any $g\in\C[t,\ODelta]$ we will denote as $C(g)$ its homogeneous component of highest degree. Then $C(f_k)=at^\alpha\Delta_1^{\beta_1}\cdots\Delta_k^{\beta_k}$ where $a\in\C$, $\alpha\geq0$ and for every $i\in\{1,\ldots,k\}$, if $g_i$ is suspicious then $\beta_i=0$. 

We consider two cases:
\begin{itemize}
\item If $g_k$ is suspicious, $\beta_k=0$, so $C(F_k)=C(f_k)^{e_k}$. When passing to $f_{k-1}$, the substitution of $\Delta_k^{e_k}=g_k$ does not affect $C(F_k)$ and cannot increase the degree of lower terms because $\deg(\Delta_k^{e_k})=\deg(g_k)$, so $C(f_{k-1})=C(F_k)$.

\item If $g_k$ is not suspicious,
\[
 C(F_k)=a^{e_k}t^{\alpha{e_k}}\Delta_1^{\beta_1e_k}\cdots\Delta_{k-1}^{\beta_{k-1}e_k} \left( \prod_{i=1}^{e_k}\gamma_k^i\right)^{\beta_k} \Delta_k^{\beta_ke_k}.
\]
Using again that the substitution of $\Delta_k^{e_k}=g_k$ does not increase degrees,
\[
 C(f_{k-1})=a^{e_k}t^{\alpha{e_k}}\Delta_1^{\beta_1e_k}\cdots\Delta_{k-1}^{\beta_{k-1}e_k} \left( \prod_{i=1}^{e_k}\gamma_k^i\right)^{\beta_k} C(g_k)^{\beta_k}.
\]
But since $g_k$ is not suspicious, $C(g_k)$ only has one monomial not involving suspicious variables. Therefore $f_{k-1}$ is not suspicious. On the other hand,
\[
\deg(f_{k-1})=\alpha{e_k}+\deg(\Delta_1)\beta_1e_k+\cdots\deg(\Delta_{k-1})\beta_{k-1}e_k+\deg(g_k)\beta_k
\]
and since $\deg(g_k)=e_k\deg(\Delta_k)$,
\[
\deg(f_{k-1}) = e_k\left( \alpha+\deg(\Delta_1)\beta_1+\cdots\deg(\Delta_k)\beta_k\right) = e_k\deg(f_k).
\]
\end{itemize}

We conclude $\deg(f_0)=\deg(R(f))=\deg(f)\cdot e_1\cdots e_m$.
\end{proof}

\begin{corollary}\label{cor-susp}
Theorem \ref{th:main} is also true if ``not guilty'' is replaced by ``not suspicious''.
\end{corollary}

\section{On missing points}\label{sec-miss}


Recall that we call missing points those $\overline{x}$ in the finite set $\overline{X(\mP)}\setminus X(\mP)$ (see Definition \ref{def:sup}), i.e. points for which there do not exist $(t,\ODelta)\in\C^{m+1}$ such that $E_i(t,\ODelta)=0$ for all $i$ and $\mP(t,\ODelta)$ is well defined and equal to {$\overline{x}$}.

To find missing points, let us apply projective elimination techniques as in \cite[Section 8.5]{CoxLittleOshea2015a}:
\begin{enumerate}
\item Consider the ideal $I\subset\C[t,\ODelta,{\overline{x}},z]$ generated by the equations defining $\DP$ (see Section \ref{sec:previos}).
\item Let $I^h\subset\C[{\overline{x}}][w,t,\ODelta,z]$ be the ideal generated by the $(w,t,\ODelta,z)$-ho\-mo\-ge\-nization of the elements of $I$ (a method based on Gr\"obner bases to compute $I^h$ can be found in \cite[Section 8.5, Proposition 10]{CoxLittleOshea2015a}). 
\item Note that $\oDP=V(I^h)$, where $\oDP$ is the projective closure of $\DP$ in $\P^{m+2}\times{\C^n}$ (see \cite[Section 8.5, Proposition 8]{CoxLittleOshea2015a}). Then, by \cite[Section 8.5, Corollary 9]{CoxLittleOshea2015a}, the projection of $\oDP\subset\P^{m+2}\times{\C^n\to\C^n}$ is exactly $\overline{X(\mP)}$, so it is Zariski closed in ${\C^n}$.
\item Therefore, the missing points are contained in the projection of the points of $\oDP$ at infinity; that is, the projection of $\oDP\cap V(w)=\oDP\setminus \DP$ onto the {$\overline{x}$} coordinates is a finite superset of the set of missing points.
\end{enumerate}

This is also a way to compute missing points for rational curves. However, it is simpler to compute the limit of the parametrization at infinity, see \cite{Sendra}. We want to replicate the latter for radical parametrizations. It is worth observing that some missing points are related to condition 1 of Theorem \ref{th:main} (the one which allows us to lift to $\AP$, see beginning of Section 3) and the others are related to condition 2 (the one which allows the point in $\AP$ to be lifted to $\DP$). 

\begin{remark}\label{rem:cota1}
The points that may not be lifted from $\overline{X(\mP)}$ to $\AP$ are those for which the leading coefficients $c_{\alpha_i}(x_i)$ of $F_i(x_i,t) = c_{\alpha_i}(x_i)t^{\alpha_i}+\cdots$ (the generator of the ideal $I(\AP)\cap\C[x_i,t]$) 
vanish simultaneously. 

There are at most $\deg(c_{\alpha_1})\cdots \deg(c_{\alpha_n})$ candidates $\overline{x}$ to be missing points. Moreover, $V(c_{\alpha_1}(x_1),\ldots,c_{\alpha_n}(x_n))\cap\overline{X(\mP)}$ is a superset of missing points due to condition 1 of Theorem \ref{th:main}.
\end{remark} 

\begin{remark}\label{rem:cota1bis}
Consider the ``Castle'' (union of all tower varieties of \cite{SendraSevillaVillarino2017}) $V(\oE)$ in the $m+1$-dimensional space of coordinates $(t,\ODelta)$. The rational map $(t,\ODelta)\mapsto \mP(t,\ODelta)$ defined in $V(\oE)$ will be called $\mP$, abusing the notation. $V(\oE)$ is of pure dimension one (\cite{SendraSevillaVillarino2017}) so, since the fundamental locus of a rational map has codimension greater than or equal to 2 in smooth algebraic varieties, $\mP$ can be extended to the whole desingularization $\mC$ of the projective closure of $V(\oE)$.

The points that are not lifted from $\overline{X(\mP)}$ to $\AP$ are the images by $\mP$ of the points of $\overline{X(\mP)}$ that are ``images'' by $\mP$ of points at infinity of $V(\oE)$. The degree of $V(\oE)$ is, at most, $\varepsilon_1\cdots \varepsilon_m$, where $\varepsilon_i=\max\{e_i,\deg(g_i)\}$ with the classical definition of degree, since it is an affine complete intersection of hypersurfaces of degrees $\varepsilon_1,\ldots,\varepsilon_m$. Therefore, the points at infinity of $V(\oE)$ are, counted with multiplicity, $\varepsilon_1\cdots \varepsilon_m$. This proves that a bound for the amount of missing points due to the lack of hypothesis 1 is $\varepsilon_1\cdots \varepsilon_m$.
\end{remark} 

\begin{example}\label{ex:cotas}
Consider $\mP=\left(\frac{\sqrt[n]{t(t-1)^{n-1}}}{t-1},\frac{\sqrt[m]{(2t-1)(t-1)^{m-1}}}{t-1}\right)=\left(\frac{\Delta_1}{t-1},\frac{\Delta_2}{t-1}\right)$, with $E_1=\Delta_1^n-t(t-1)^{n-1}$, $E_2=\Delta_2^m-(2t-1)(t-1)^{m-1}$. Note that $\mP$ can be simplified to $\left(\sqrt[n]{\frac{t}{t-1}},\sqrt[m]{\frac{2t-1}{t-1}}\right)$, and this imposes a different way to choose $\ODelta$ and $\oE$.

It is easy to see that $\overline{X(\mP)}$ is defined by the equation $x^n-y^m+1=0$. The polynomials $F_1$ and $F_2$ of Remark \ref{rem:cota1} are
\[ \begin{array}{l} F_1(x,t)=x^n(t-1)-t=(x^n-1)t-x^n,\\
 F_2(y,t)=y^m(t-1)-2t+1=(y^m-2)t-y^m+1. \end{array}\]
Then {$\deg(c_{\alpha_1})=n$ and $\deg(c_{\alpha_2})=m$}. The possible missing points are those whose first coordinate is an $n$-th root of $1$ and second coordinate is an $m$-th root of $2$. One can check that, indeed, all those points cannot be reached for any value of $t$, so we have $mn$ missing points in total.  Therefore, the bounds in Remarks \ref{rem:cota1} and \ref{rem:cota1bis} are sharp. 

Note that $F_1(x,1)=-1$, hence no point in $\AP$ has coordinate $t=1$, so there are no additional missing points.
\end{example} 

\begin{remark}
In Example \ref{ex:cotas}, it is not difficult to check that the parametrization is proper (i.e. generically injective). In general, when $X(\mP)$ is irreducible, one could deduce properness if the greatest common divisor of $F_1$ and $F_2$, seen as polynomials with coefficients in the field of rational functions of $X(\mP)$, is the power of a degree one polynomial. 
\end{remark}

\begin{remark}\label{rem:cota2}
The points that cannot be lifted from $\AP$ to $\DP$ are among those whose coordinates $(t,\ODelta)$ satisfy the equations $E_1,\ldots,E_m,p_i,q_i$ for some $i=1,\ldots, n$. Taking their projections on the space of $\ox$ coordinates, this means that the number of missing points due to hypothesis 2 failure is less than or equal to the number of solutions of those algebraic systems of equations.
\end{remark}

\noindent \textbf{Acknowledgements.} The authors are partially supported by the grant PID2020-113192GB-I00/AEI/ 10.13039/501100011033 (Mathematical Visualization: Foundations, Algorithms and Applications) from the Spanish State Research Agency (Ministerio de Ciencia e Innovación). J. Caravantes and C. Villarino belong to the Research Group ASYNACS (Ref. CT-CE2019/683). D. Sevilla is a member of the research group GADAC and is partially supported by Junta de Extremadura and Fondo Europeo de Desarrollo Regional (GR21055).


\end{document}